\documentclass[12pt]{article}
\usepackage{amsfonts}
\usepackage{bm}
\usepackage{amsmath, amsthm, amsfonts, amssymb,  color}
\usepackage{mathrsfs}
\usepackage{graphicx}
\usepackage{appendix}
\newtheorem{thm}{Theorem}[section]
\newtheorem{cor}[thm]{Corollary}
\newtheorem{lem}[thm]{Lemma}

\newtheorem{exa}[thm]{Example}
\theoremstyle{definition}
\newtheorem{defn}{Definition}[section]
\newcommand{\scr}[1]{\mathscr #1}
\definecolor{wco}{rgb}{0.5,0.2,0.3}

\numberwithin{equation}{section} \theoremstyle{remark}
\newtheorem{rem}{Remark}[section]

\newcommand{\ua}{\uparrow}

\renewcommand{\hat}{\widehat}

\title{{\bf Almost Sure Asymptotic Stability for Regime-Switching Diffusions}
}
\author{Junhao Hu,\thanks{College of Mathematics and Statistics,
South-Central University for Nationalities, Wuhan, 430074, China,
junhaohu74@163.com} \and Jianhai Bao,\thanks{Department of
Mathematics, Central South University, Changsha, Hunan, 410083,
China, jianhaibao13@gmail.com} \and Chenggui Yuan\thanks{Department
of Mathematics, Swansea University, Singleton Park, SA2 8PP, UK,
C.Yuan@swansea.ac.uk}}

\begin{document}
\def\R{\mathbb R}  \def\ff{\frac} \def\ss{\sqrt} \def\B{\mathbf
B}
\def\N{\mathbb N} \def\kk{\kappa} \def\m{{\bf m}}
\def\dd{\delta} \def\DD{\Dd} \def\vv{\varepsilon} \def\rr{\rho}
\def\<{\langle} \def\>{\rangle} \def\GG{\Gamma}
  \def\nn{\nabla} \def\pp{\partial} \def\EE{\scr E}
\def\d{\text{\rm{d}}} \def\bb{\beta} \def\aa{\alpha} \def\D{\scr D}
  \def\si{\sigma} \def\ess{\text{\rm{ess}}}
\def\beg{\begin} \def\beq{\begin{equation}}  \def\F{\scr F}
\def\Ric{\text{\rm{Ric}}} \def\Hess{\text{\rm{Hess}}}
\def\e{\text{\rm{e}}} \def\ua{\underline a} \def\OO{\Omega}  \def\oo{\omega}
 \def\tt{\tilde} \def\Ric{\text{\rm{Ric}}}
\def\cut{\text{\rm{cut}}} \def\P{\mathbb P} \def\ifn{I_n(f^{\bigotimes n})}
\def\C{\scr C}      \def\aaa{\mathbf{r}}     \def\r{r}
\def\gap{\text{\rm{gap}}} \def\prr{\pi_{{\bf m},\varrho}}  \def\r{\mathbf r}
\def\Z{\mathbb Z} \def\vrr{\varrho} \def\l{\lambda}
\def\L{\scr L}\def\Tt{\tt} \def\TT{\tt}\def\II{\mathbb I}
\def\i{{\rm in}}\def\Sect{{\rm Sect}}\def\E{\mathbb E} \def\H{\mathbb H}
\def\M{\scr M}\def\Q{\mathbb Q} \def\texto{\text{o}} \def\LL{\Lambda}
\def\Rank{{\rm Rank}} \def\B{\scr B} \def\i{{\rm i}} \def\HR{\hat{\R}^d}
\def\to{\rightarrow}\def\l{\ell}\def\lf{\lfloor}\def\rf{\rfloor}
\def\8{\infty}\def\ee{\epsilon} \def\Y{\mathbb{Y}} \def\lf{\lfloor}
\def\rf{\rfloor}\def\3{\triangle} \def\O{\mathcal {O}}
\def\SS{\mathbb{S}}\def\ta{\theta}\def\h{\hat}

\def\la{\langle}\def\ra{\rangle}

\def\trace{\text{\rm{trace}}}
\renewcommand{\bar}{\overline}
\def\Y{\mathbb{Y}}
\renewcommand{\tilde}{\widetilde}
\date{}

\maketitle

\begin{abstract}
In this paper, we discuss long-time behavior of sample paths for a
wide range of regime-switching diffusions.  Firstly, almost
sure asymptotic stability is concerned  (i) for regime-switching
diffusions with finite state spaces by the Perron-Frobenius theorem,
and, with regard to the case of reversible Markov chain, via the
principal eigenvalue approach; (ii) for regime-switching diffusions
 with countable state spaces by means of a finite partition trick
and an M-Matrix theory. We then apply our theory to study the stabilization for  linear switching models. Several examples are given to demonstrate our theory.

 \end{abstract}

\section{Introduction}
Stochastic stability of stochastic differential equations (SDEs) has
been well developed (see e.g. the monographs \cite{Kh12,M94,M07}.
For a regime-switching diffusion, we mean a diffusion process in a
random environment characterized by a Markov chain. Its state vector
is a pair $(X_t,\LL_t)$, where the continuous component $X_t$ is
referred to as the state, while the discrete component $\LL_t$ is
regarded as the mode. Regime-switching diffusions have been widely
applied in control problems, storage modeling, neutral activity,
biology and mathematical finance (see e.g. \cite{MY06,YZ}). In
particular, an important issue in the study of regime-switching
diffusions is concerned with stability. In the past three decades,
stochastic stability of regime switching diffusions has also
received great attention (see e.g. the monographs \cite{MY06,YZ}).
\cite[Exampe 5.45, p.223]{MY06} reveals that $X_t$ is stable
although some of the subsystems are not.  So, in most cases,
stability analysis for regime-switching processes may be markedly
differently different from that of SDEs with regime switching. So
far, the work on stability analysis for regime-switching processes
focus on moment exponential stability \cite{Kh07,M99,M00}, almost
sure exponent stability \cite{MSY,XY}, asymptotically stable in
probability \cite{Kh07,SX14,YZ}, stability in distribution
\cite{YM03,YMZ}, to name a few. For ergodic property, strong Feller,
 transience and recurrence for regime-switching diffusions,  we would like to  refer
 to \cite{BGM,SX,SJ,S14,SX14,Xi} and references therein.

\smallskip
For most existing results, condition  to guarantee stability are
irrespective of stationary distribution of Markov chain involved
(see e.g. \cite{M99,M00,MSY,XY,YM03,YMZ,YZ}).  Moreover, note that
the vast majority of stability analysis focus on regime-switching
diffusions with finite state spaces (see e.g.
\cite{M99,M00,MSY,XY,YM03,YMZ,YZ}). In this paper, {\it under some
new conditions} (see Theorems \ref{T-1}, \ref{M-1} and \ref{coun}
for more details), we shall discuss stability analysis of sample
paths for a class of regime-switching diffusions which may admit
infinite state spaces. The content of this paper is arranged as
follows. In Section \ref{sec3}, for regime-switching diffusions with
finite state spaces, under an ``averaging condition''we study almost
sure asymptotic stability of sample path (see Theorem \ref{T-1})  by
the Perron-Frobenius theorem. In particular,  Theorem \ref{T-1}
improves greatly some existing results (see e.g. \cite[Theorem
3.1]{ZHZ12} and \cite[Theorem 5.29, p.192]{MY06}). For more details,
please refer to Examples \ref{Ex1} and \ref{Ex2}. Section \ref{sec4}
is devoted to diffusion processes with reversible Markov chains. For
such special case, the principal eigenvalue approach is adopted to
deal with almost sure asymptotic stability (see Theorem \ref{M-1}),
where Theorem \ref{M-1} cannot be covered by Theorem \ref{T-1}
 as Example \ref{ex4.3} shows. Note that Theorem \ref{T-1} is dependent
 on the explicit formula of stationary distribution of Markov chain,
 and Theorem \ref{eigen} need  the principal eigenvalue to be
 attainable. Accordingly, Theorems \ref{T-1} and \ref{M-1} seem hard
 to be generalized to the counterpart with infinite
 state space. Nevertheless, for a regime-switching diffusion  with an infinite
 state space, in Section \ref{sec5}, by a finite
 partition trick and an M-matrix theory, we proceed to investigate
 almost sure asymptotic stability.

\section{Problem Setup}
For each integer $n>0$, let $(\R^n,\<\cdot,\cdot\>,|\cdot|)$ be the
$n$-dimensional  Euclidean space and $\R^n\otimes\R^m$  the totality
of all $n\times m$ matrices  endowed with the Frobenius norm
$\|\cdot\|$. Let $\{W_t\}_{t\ge0}$ be an $m$-dimensional Brownian
motion defined on the probability space $(\OO, \F, \P)$
 with a filtration $\{\F_t\}_{t\ge0}$ satisfying
the usual conditions (i.e., $\F_t=\F_{t+}:=\bigcap_{s>t}\F_s$ and
$\F_0$ contains all $\P$-null sets).   If $A$ is a vector or matrix,
its
transpose is denoted by $A^T$. Let ${\bf0} $  be a zero vector, for $\xi=(\xi_1,\cdots,\xi_n)^T\in\R^n$,
$\xi\gg{\bf0} (\mbox{resp. }  \xi\ll{\bf0})$  means
each component $\xi_i>0 (\mbox{resp. } \xi_i<0)$, $i=1,\cdots,n.$
$C^2(\R^n;\R_+)$ stands for the family of all nonnegative functions
$f:\R^n\times\R_+\mapsto\R_+$  which are continuously twice
differentiable.

\smallskip

In this paper, we focus on a regime-switching diffusion process
$(X_t,\LL_t)$, where $\{X_t\}_{t\ge0}$ satisfies an SDE on
$(\R^n,\<\cdot,\cdot\>,|\cdot|)$
\begin{equation}\label{eq1}
\d X_t=b(X_t, t,\LL_t)\d t+\si(X_t, t,\LL_t)\d W_t,~t>0,~
X_0=x_0,~\LL_0=i_0,
\end{equation}
and
 $\{\LL_t\}_{t\ge0}$, independent of the Brownian motion $\{W_t\}_{t\ge0}$, is a right-continuous Markov
 chain
defined on the probability space $(\OO, \F, \P)$ with the state
space $\mathbb{S}:=\{1,2,\cdots,N\}$ for some  $1\le N\le\8$ and the
transition rules specified by
\begin{equation}\label{love}
\P(\LL_{t+\delta}=j|\LL_t=i)=
\begin{cases}
q_{ij}\delta+o(\delta),~~~~~~~~i\neq j,\\
1+q_{ii}\delta+o(\delta),~~~i=j
\end{cases}
\end{equation}
for sufficiently small $\delta>0 $. Here $q_{ij} \ge0$ is the
transition rate from $i$ to $j$ if $i\neq j$ while
$
  q_{ii}=- \sum_{i\neq j} q_{ij}.
$ We assume that the $Q$-matrix $Q:=(q_{ij})_{N\times N}$ is
irreducible   so that the Markov chain $\{\LL_t\}_{t\ge0}$ has a
unique stationary distribution $\mu:=({\mu_1,\cdots,\mu_N})$ which
can be determined by solving the following linear equation $$\mu
Q={\bf0}$$ subject to $\mu_1+\cdots+\mu_N=1 ~\mbox{ and }~\mu_i>0,\,
i\in\mathbb{S}.$

\smallskip

It is well known that a continuous-time Markov chain
$\{\LL_t\}_{t\ge0}$ with generator $Q=(q_{ij})_{N\times N}$ can be
described in the following manner by using the Poisson random
measure. Let
\begin{equation*}
h(i,z)=\sum_{j\in\mathbb{S}\setminus i}(j-i){\bf
1}_{z\in\triangle_{ij}},
\end{equation*}
where $\triangle_{ij}$ are the consecutive (with respect to (w.r.t.)
the lexicographic order on $\mathbb{S}\times\mathbb{S}$ left-closed,
right-open intervals on $\R_+$, each having length $q_{ij}$, with
$\triangle_{12}=[0,q_{12})).$ Then
\begin{equation*}
\d \LL_t=\int_\R h(X_{t-},z)N(\d t,\d z), ~~~\LL_0=i_0,
\end{equation*}
where $N(\d t,\d z)$ is a Poisson random measure with intensity $\d
t\times\d z.$

\smallskip
Let $(X_t,\LL_t)$  be the regime-switching diffusion process
determined by \eqref{eq1} and \eqref{love}. For each fixed
environment $i\in\mathbb{S}$, the corresponding diffusion
$X_t^{(i)}$ is defined by
\begin{equation}\label{*}
\d X_t^{(i)}=b(X_t^{(i)}, t,i)\d t+\si(X_t^{(i)}, t,i)\d
W_t,~~~t>0,~~~X_0^{(i)}=x_0.
\end{equation}
Then, the infinitesimal generator associated with \eqref{*} is given
by
\begin{equation*}
\mathscr{L}^{(i)}_tV(x):=\<\nn
V(x),b(x,t,i)\>+\ff{1}{2}\mbox{trace}(\si^T(x,t,i)\nn^2V(x)\si(x,t,i)),~V\in
C^2(\R^n;\R_+),
\end{equation*}
where $\nn$ and $\nn^2$ stand for the gradient and Hessian operators
respectively.

\smallskip
 Throughout the paper, we assume that
\begin{enumerate}
\item[\textmd{({\bf H1})}] $b:\R^n\times\R_{+}\times\mathbb{S}\mapsto\R^n$ and
$\si:\R^n\times\R_{+}\times\mathbb{S}\mapsto\R^n\otimes\R^m$ satisfy
the local Lipschitz condition with respect to (w.r.t.) the first
variable, i.e., for each $ k>0$,
 there exists an $L_k>0$ such that
\begin{equation*}
|b(x,t,i)-b(y,t,i)|+\|\si(x,t,i)-\si(y,t,i)\|\le L_k|x-y|,~~~|x|
\vee |y| \le k,
\end{equation*}
for all $(t,i)\in \R_{+}\times \mathbb{S}$. Moreover,
\begin{equation*}
\sup_{t\ge0,i\in\mathbb{S}}\{|b(0,t,i)|+\|\si(0,t,i)\|\}<\8.
\end{equation*}

\item[\textmd{({\bf H2})}] For each fixed $i\in \mathbb{S}$,  there exists a function $V\in C^2(\R^n;\R_{+})$ and $\beta_i\in\R$ such that $\lim_{|x|\to\8}V(x)=\8, V(0)=0$  and
\begin{equation}\label{cond-1}
\mathscr{L}^{(i)}_tV(x)\le  \beta_i V(x), ~~~(x,t)\in
\R^n\times\R_{+}.
\end{equation}
\end{enumerate}

 Under ({\bf H1}) and
({\bf H2}), by the classical Khasminskii approach (see e.g.
\cite[Theorem 3.5, p.75]{Kh12}), \eqref{eq1} admits a unique strong
solution $\{X_t(x_0,i_0)\}_{t\ge0}$ to highlight the initial data
$X_0=x_0$ and $\LL_0=i_0$.

\begin{defn}
{\rm The solution of  \eqref{eq1} is said to be almost surely
asymptotically stable if, for all $x_0\in \R^n$ and
$i_0\in\mathbb{S}$,
\begin{equation*}
\P\left(\lim_{t\to \8} |X_t(x_0,i_0)|=0 \right)=1.
\end{equation*}
}
\end{defn}

The lemma (see e.g. \cite[Thorem 7, p.139]{LS})) below, which is
concerned with long-time behavior of  nonnegative semi-martingales,
plays a crucial role in the following  stability analysis of  sample
path for \eqref{eq1}.

\begin{lem}\label{S-C}
{\rm
 Let $\{A_1(t)\}_{t\ge0}$, $\{A_2(t)\}_{t\ge0}$ be two continuous adapted increasing processes with $A_1(0)=A_2(0)=0$
a.s.,  $M(t)$ a real-valued continuous local martingale with
$M(0)=0$ a.s., and    $\zeta$  a nonnegative $\F_0$-measurable
random variable such that $\E\zeta<\8$. Define
\begin{equation*}
X(t)=\zeta+A_1(t)-A_2(t)+M(t),~~~ t\ge0.
\end{equation*}
If $X(t)$ is nonnegative, then
\begin{equation*}
\left\{ \lim_{t\to \8}A_1(t)<\8 \right \} \subset  \left\{
\lim_{t\to \8}X(t)<\8 \right\}\bigcap\left \{ \lim_{t\to
\8}A_2(t)<\8 \right\} ~~ \mbox{a.s.},
\end{equation*}
where $C\subset D$ a.s. means $\P(C\cap D^c)=0$. In particular, if $
\lim_{t\to \8}A_1(t)<\8 $ a.s., then,
\begin{equation*}
\lim_{t\to\8}X(t)<\8, ~~\lim_{t\to\8}A_2(t)<\8,~~\mbox{ and
}~~-\8<\lim_{t\to\8}M(t)<\8,~~~\mbox{a.s.}
\end{equation*}
 }
\end{lem}

\section{Almost Sure Asymptotic Stability}\label{sec3}

Our main result in this section is stated as below. Throughout this section, we assume that the Markov state space is finite, i.e. $N<\8.$
\begin{thm}\label{T-1}
{\rm Let  assumptions  ({\bf H1}) and  ({\bf H2}) hold.
Suppose further that
\begin{equation}\label{cond-2}
\sum_{i=1}^N\mu_i\beta_i <0.
\end{equation}
Then,  the solution of \eqref{eq1} is  almost surely asymptotically
stable. }
\end{thm}

\begin{rem}
{\rm Equation \eqref{eq1} is said to be attractive ``in average'' if
\eqref{cond-2} holds (see e.g. \cite{BGM}). We call \eqref{cond-2}
an ``averaging condition''. On the other hand,  in Theorem
\ref{T-1}, it is worth to pointing out that, for each $i\in
\mathbb{S}$, $\bb_i$ need not to be negative.
 }
\end{rem}

Before the proof of  Theorem \ref{T-1}, let us recall Proposition
4.2 due to Bardet et al. \cite{BGM},    which is stated as below for
convenience.

\begin{lem}\label{P-F}
{\rm  For any $p>0$, let
\begin{equation}\label{r4}
\mbox{diag}(\bb):=\mbox{diag}(\bb_1,\cdots,\bb_N),~~~~ Q_p:=Q+
p\,\mbox{diag}(\bb),~~\eta_p:=-\max_{\gamma\in\mbox{spec}(
Q_p)}\mbox{Re}\gamma,
\end{equation}
where $Q$ is the $Q$-matrix of the Markov chain
$\{\LL(t)\}_{t\ge0},$ and $\mbox{spec}(Q_p)$ denotes the spectrum of
$Q_p.$ Under \eqref{cond-2}, one has
\begin{enumerate}
\item[(i)] $\eta_p>0$ if $\max_{i\in\mathbb{S}}\bb_i\le0;$

\item[(ii)] $\eta_p>0$ for  $p<\aa$, where $\aa\in(0,\min_{\bb_i>0,i\in\mathbb{S}}\{-q_{ii}/\bb_i\})$ with $\max_{i\in\mathbb{S}}\bb_i>0.$
\end{enumerate}
}
\end{lem}

\smallskip

We are now in the position to complete the argument of Theorem
\ref{T-1}.

\smallskip

\noindent {\it Proof of Theorem \ref{T-1}.} Let
$Q_{p,t}:=\e^{tQ_p}$, where $Q_p$ is defined as in \eqref{r4}. Then
the spectral radius $\mbox{Ria}( Q_{p,t})$ of $Q_{p,t}$ equals to
$\e^{-\eta_p}$. Since all coefficients of $  Q_{p,t}$ are positive,
the Perron-Frobenius theorem (see e.g. \cite[p.6]{CM}) yields that
$-\eta_p$ is a simple eigenvalue of $Q_p.$ Moreover, note that the
eigenvalue of $Q_{p,t}$ corresponding to $\e^{-\eta_p}$ is also an
eigenvalue of $Q_p$ corresponding to $-\eta_p.$ The Perron-Frobenius
theorem (see e.g. \cite[p.6]{CM}) ensures that there exists an
eigenvector $\xi^{(p)}=(\xi_1^{(p)},\cdots,\xi_N^{(p)})\gg{\bf0}$ of
$Q_p$ corresponding to $-\eta_p.$ Now, by Lemma \ref{P-F} above,
there exists some $ p_0>0$ such that $\eta_p>0$ for any $0< p< p_0.$
In what follows, fix a $p$ with $0< p\le\min\{1, p_0\}$ and the
corresponding eigenvector $\xi^{(p)}\gg{\bf0}$. Then we obtain that
\begin{equation}\label{eq2}
Q_{p}\xi^{(p)}=-\eta_{ p}\xi^{(p)}\ll{\bf0}.
\end{equation}
For notation simplicity, in what follows,  we write
$\{X_t\}_{t\ge0}$ in lieu of $\{X_t(x_0,i_0)\}_{t\ge0}$. By the
It\^o formula, it follows that
\begin{equation*}
\begin{split}
V^p(X_t)\xi^{(p)}_{\LL_t}&=V^p(x_0)\xi^{(p)}_{i_0}+\int_0^t\{(\mathscr{L}^{(\LL_s)}_sV^p(X_s))\xi^{(p)}_{\LL_s}+V^p(X_s)(Q\xi^{(p)})(\LL_s)\}\d s\\
&\quad+\int_0^t\xi^{(p)}_{\LL_s}\<\nn V^p(X_s),\si(X_s,s,\LL_s)\d W_s\>\\
&\quad+\int_0^t\int_\R V^p(X_s)\{\xi^{(p)}_{i_0+h(\LL_s,z)}-\xi^{(p)}_{\LL_s}\}\tt N(\d s,\d z)\\
&=:J_1(t)+J_2(t)+J_3(t)+J_4(t),
\end{split}
\end{equation*}
where $(Q\xi^{(p)})(i), i\in\mathbb{S}$, denotes the $i$-th entry of
the vector $Q\xi^{(p)}$, and $\tt N(\d t,\d z):=N(\d t,\d z)-\d t\d
z$. Due to $p\in(0,1]$, observe from \eqref{cond-1} and
\eqref{cond-2} that
\begin{equation*}
\begin{split}
J_2(t)
&=\int_0^t\Big\{pV^{p-1}(X_s)(\mathscr{L}^{(\LL_s)}_sV(X_s))\xi^{(p)}_{\LL_s}+V^p(X_s)(Q\xi^{(p)})(\LL_s)\\
&\quad+\ff{p(p-1)}{2}V^{p-2}(X_s)|\si^T(X_s,s,\LL_s)\nn
V(X_s)|^2\Big\}\d s\\
&\le\int_0^t\{p\bb_{\LL_s}\xi^{(p)}_{\LL_s}+(Q\xi^{(p)})(\LL_s)\}V^p(X_s)\d s\\
&=\int_0^t(Q_p\xi^{(p)})(\LL_s)V^p(X_s)\d s\\
&=-\eta_p\int_0^t\xi^{(p)}_{\LL_s}V^p(X_s)\d s.
\end{split}
\end{equation*}
Hence, we arrive at
\begin{equation*}
\begin{split}
V^p(X_t)\xi^{(p)}_{\LL_t}&\le
V^p(x_0)\xi^{(p)}_{i_0}-\eta_p\int_0^t\xi^{(p)}_{\LL_s}V^p(X_s)\d
s+J_3(t)+J_4(t).
\end{split}
\end{equation*}
Note that $J_3(t)$ and $J_4(t)$ are local martingales. Set
\begin{equation*}
Y(t):=V^p(x)\xi^{(p)}_{i_0}-\eta_p\int_0^t\xi^{(p)}_{\LL_s}V^p(X_s)\d
s+J_3(t)+J_4(t).
\end{equation*}
It is easy to see that $V^p(X_t)\xi^{(p)}_{\LL_t}\le Y(t)$,
$\mbox{a.s.},$ and $Y(t)$ is a nonnegative semi-martingale. Thus, by
Lemma \ref{S-C}, we infer that
\begin{equation}\label{*2}
\lim_{t\to\8}V(X_t)<\8~~\mbox{a.s.}~~~~\mbox{ and
}~~~~\int_0^\8V^p(X_t)\d t<\8 ~~\mbox{a.s.}
\end{equation}
This,
 together with $\lim_{|x|\to\8}V(x)=\8,$  further yields that
\begin{equation*}
\sup_{0\le
t<\8}|X_t|<\8~~~\mbox{a.s.},
\end{equation*}
and,
\begin{equation*}
\liminf_{t\to\8}V(X_t)=0~~\mbox{a.s.}.
\end{equation*}
Moreover, we can also claim that there exists $\OO_0\subset\OO$ with
$\P(\OO_0)=1$ such that
\begin{equation}\label{*021}
\lim_{t\to\8}V(X_t(\oo))=0,~~~\oo\in\OO_0.
\end{equation}
Then, by a contraction argument, we derive that
\begin{equation}\label{*022}
\lim_{t\to \8}X_t(\oo)=0,~~~\oo\in\OO_0.
\end{equation}
The desired assertion is therefore complete. For more details on
argument of \eqref{*021} and \eqref{*022}, please refer to Appendix A.

\smallskip

In ({\bf H2}), taking $V(x)=|x|^2$, we derive the following
corollary.
\begin{cor}\label{cor-1}
{\rm Let  ({\bf H1}) hold.  Assume that, for each
$i\in\mathbb{S}$, there exists $\beta_i\in\R$ such that
\begin{equation*}
2\langle x, b(x,t,i)\rangle+\|\si(x,t,i)\|^2 \le  \beta_i |x|^2,
~~~\forall (x,t)\in \R^n\times\R_{+}.
\end{equation*}
Assume further that $$ \sum_{i=1}^N\mu_i\beta_i <0.$$ Then, the
solution of  \eqref{eq1} is almost surely asymptotically stable. }
\end{cor}

We assume that
\begin{enumerate}
\item[\textmd{({\bf H$2^\prime$})}] For each $i\in\mathbb{S}$, there exist   $V\in C^2(\R^n;\R_{+})$ with compact level
sets, $\gamma\in L^1(\R_{+};\R_{+})$, and $\beta_i\in\R$ such that
\begin{equation}\label{cond-4}
 \mathscr{L}^{(i)}_tV(x)\le \gamma_t+\beta_i V(x), ~~~ (x,t)\in \R^n\times\R_{+}.
\end{equation}
\end{enumerate}

\begin{thm}\label{M-1}
{\rm  Let  ({\bf H1}) and ({\bf H$2^\prime$}) hold. Assume
further that
\begin{equation}\label{*3}
\min_{i\in\mathbb{S}}\{-q_{ii}/\beta_i: \beta_i>0\}>1,
\end{equation}
and
\begin{equation}\label{*4}
 \sum_{i=1}^N\mu_i\beta_i <0.
\end{equation}
Then, the solution of  \eqref{eq1} is almost surely asymptotically
stable. }
\end{thm}

\begin{proof}
 Let
\begin{equation*}
Q_1:=Q+\mbox{diag}(\bb_1,\cdots,\bb_N),
\end{equation*}
where $\bb_i\in\R$ such that \eqref{*4}. Due to \eqref{*3}, upon
following an argument of \eqref{eq2}, there exist $\eta>0$ and a
vector $\xi=(\xi_1, \ldots, \xi_N)\gg {\bf0}$ such that
\begin{equation}\label{*5}
Q_1\xi=-\eta\xi\ll{\bf0}.
\end{equation}
By It\^o's formula, we derive from \eqref{cond-4} and \eqref{*5}
that
\begin{equation*}
\begin{split}
V(X_t)\xi_{\LL_t}&=V(x_0)\xi_{i_0}+\int_0^t\{(\mathscr{L}_s^{(\LL_s)}V(X_s))\xi_{\LL_s}+V(X_s)(Q\xi)(\LL_s)\}\d
s+\Lambda_1(t)+\Lambda_2(t)\\
&\le
V(x_0)\xi_{i_0}+\int_0^t\{\gamma_s\xi_{\LL_s}+(\bb_{\LL_s}\xi_{\LL_s}+(Q\xi)(\LL_s))V(X_s)\}\d
s+\Lambda_1(t)+\Lambda_2(t)\\
&=
V(x_0)\xi_{i_0}+\int_0^t\{\gamma_s\xi_{\LL_s}+(Q_1\xi)(\LL_s)V(X_s)\}\d
s+\Lambda_1(t)+\Lambda_2(t)\\
&\le V(x_0)\xi_{i_0}+\max_{i\in\mathbb{S}}\xi_i\int_0^t\gamma_s\d
s-\eta\int_0^tV(X_s)\xi_{\LL_s}\d s+\Lambda_1(t)+\Lambda_2(t),
\end{split}
\end{equation*}
where
\begin{equation*}
\Lambda_1(t):=\int_0^t\xi_{\LL_s}\<\nn V(X_t),\si(X_s,s,\LL_s)\d
W_s\>,~~\Lambda_2(t):=\int_0^t\int_\R
V(X_s)\{\xi_{i_0+h(\LL_s,z)}-\xi_{\LL_s}\}\tt N(\d s,\d z).
\end{equation*}
Note that $\gamma\in L^1(\R_{+};\R_{+})$. Thus, the desired
assertion follows by carrying out an argument of Theorem \ref{T-1}.
\end{proof}

\smallskip

In ({\bf H2'}), taking $V(x)=|x|^2$, we deduce the  corollary below.
\begin{cor}\label{cor-2}
{\rm Let  ({\bf H1}) hold.  Assume that, for each
$i\in\mathbb{S}$, there exists $\gamma\in L^1(\R_{+};\R_{+})$, and
$\beta_i\in\R$ such that
\begin{equation*}
2\langle x, b(x,t,i)\rangle+\|\si(x,t,i)\|^2 \le \gamma_t+ \beta_i
|x|^2, ~~~\forall (x,t)\in \R^n\times\R_{+}.
\end{equation*}
Assume further that $$ \sum_{i=1}^N\mu_i\beta_i <0.$$ Then, the
solution of  \eqref{eq1} is almost surely asymptotically stable. }
\end{cor}

In the sequel, we provide two examples to demonstrate applications
of our theory.

\begin{exa}\label{Ex1}
{\rm Let $\{W_t\}_{t\ge0}$ be a scalar Brownian motion. Consider a
regime-switching diffusion process $(X_t,\LL_t)$, where
$\{X_t\}_{t\ge0}$ obeys a scalar SDE
\begin{equation}\label{eq8}
\d X_t= b_{\LL_t}X_t \d t+\si_{\LL_t}X_t \d
W_t,
\end{equation}
with initial value $(X_0, \LL_0),$  and $\{\LL_t\}_{t\ge0}$, independent of $\{W_t\}_{t\ge0}$, is a
right-continuous Markovian chain taking values in
$\mathbb{S}=\{1,2\}$ with generator
$$
  Q= \left(
       \begin{array}{cc}
         -u & u \\
         v & -v \\
       \end{array}
     \right)
$$
for $u,v>0$.

\smallskip

In \eqref{eq8}, set $b_1=1,b_2=-2,\si_1=1,\si_2=1$. Note that the
solution of SDE
\begin{equation*}
\d X_t=X_t\d t+X_t\d W(t),~~t>0,~~~X_0=x_0
\end{equation*}
is unstable.  It is easy to see that the unique stationary
distribution of $\{\LL_t\}_{t\ge0}$ is
$(\mu_1,\mu_2)=(\frac{v}{u+v}, \frac{u}{u+v}) $. A straightforward
calculation gives that
\begin{equation*}
\mathscr{L}_t^{(1)}(|x|^2)=3|x|^2~~~\mbox{ and
}~~~\mathscr{L}_t^{(2)}(|x|^2)=-3|x|^2,~~~~x\in\R.
\end{equation*}
So $\beta_1=3, \beta_2=-3$.  Thus, by Corollary \ref{cor-2}, the
solution of \eqref{eq8} is almost surely asymptotically stable if
$0<v<u. $

\smallskip

However, in order for  the solution of \eqref{eq8}
is almost surely asymptotically stable,  the condition in  \cite[Theorem 3.1]{ZHZ12} is  $u>6$ and $v\in(0,3)$.
Therefore, the present Theorem \ref{T-1} improves some  existing results in
certain sense.}
\end{exa}

\begin{exa}\label{Ex2}
{\rm Let $\{W_t\}$ be a scalar Brownian motion. Consider a
regime-switching diffusion process $(X_t,\LL_t)$, in which
$\{X_t\}_{t\ge0}$ satisfies a scalar SDE
\begin{equation}\label{eq9}
\d X_t=b(X_t,t,\LL_t) \d t+\si(X_t,t,\LL_t) \d
W(t),
\end{equation}
with initial value $(X_0, \LL_0),$ and $\{\LL_t\}_{t\ge0}$,   independent of $\{W_t\}_{t\ge0}$,  is a
right-continuous Markovian chain taking values in
$\mathbb{S}:=\{0,1,2\}$ with the generator
\begin{equation}\label{*7}
Q= \left(\begin{array}{ccc}
  -(3+\nu) & \nu & 3\\
  1 & -3 & 2\\
  1 &2 &-3
  \end{array}
  \right)
  \end{equation}
for some $\nu\ge0.$ In \eqref{eq9}, for any $(x,t)\in\R\times\R_+,$
let
\begin{align*}
& b(x,t,0)=\ff{x}{4}, ~~~ ~~~~~~~~~~~~~~~~~~~~~~~\si(x,t,0)=(1+t)^{-1},\\
& b(x,t,1)=\sin x/(1+t) , ~~~ ~~~~~~~~~\si(x,t,1)=\ff{x}{2} ,\\
& b(x,t,2)= \e^{-t}-5x-2x^3, ~~~ ~~~~~~\si(x,t,2)= x \sin t.
\end{align*}
Note that the unique stationary distribution of $\{\LL_t\}_{t\ge0}$
is
\begin{equation*}
\mu=(\mu_0,\mu_1,\mu_2)=\Big(\ff{5}{20+5\nu},\ff{6+3\nu}{20+5\nu},\ff{9+2\nu}{20+5\nu}\Big).
\end{equation*}
Next, by the fundamental inequality:
$2ab\le\vv a^2+b^2/\vv$ for $a,b\in\R,\vv>0$, it follows that
\begin{equation*}
\begin{split}
 \mathscr{L}^{(0)}_t(|x|^2)=\ff{1}{(1+t)^2}+\ff{|x|^2}{2},~\mathscr{L}^{(1)}_t(|x|^2)\le \ff{4}{(1+t)^2}+\ff{|x|^2}{2},
 \end{split}
\end{equation*}
and
\begin{equation*}
\begin{split}
\mathscr{L}_t^{(2)}(|x|^2)\le\e^{-2t}-8|x|^2.
\end{split}
\end{equation*}
Hence, $\beta_0=\beta_1=\ff{1}{2}$, $\bb_2=-8$, and $\gamma_t=4
(1+t)^{-2}+\e^{-2t}$. For \eqref{eq9}, it is trivial to see that
\eqref{*3} and \eqref{*4} hold respectively. By Theorem \ref{M-1},
the solution of Eq.\eqref{eq9} is almost surely asymptotically
stable.

\smallskip

 Whereas, observe  that $b(0,t,2)=\e^{-t}\neq0$ and
$b(x,t,2)$
 don't satisfy
the linear growth condition so that \cite[Theorem 5.29, p.192]{MY06}
does not apply to \eqref{eq9}. }
\end{exa}

\section{Almost Sure Asymptotic Stability: Reversible
Case}\label{sec4} In the last section, we investigate almost sure
asymptotic stability for the regime-switching diffusion process
determined by \eqref{eq1} and \eqref{*}, where the Markov chain
$\{\LL_t\}_{t\ge0}$ need not to be reversible, i.e.,
$\pi_iq_{ij}=\pi_jq_{ji},i,j\in\mathbb{S}$, for some probability
measure $\pi:=(\pi_1,\cdots,\pi_N)$. While,  if $\{\LL_t\}_{t\ge0} $
with finite state space, i.e., $N<\8$,  is reversible, under another
new condition the long-time behavior of sample path for \eqref{eq1}
can also be discussed as Theorem \ref{eigen} below shows.

\smallskip

To begin with, we need to introduce some notation. Throughout this
section, we always assume that $N<\8$.
 Let
\begin{equation*}
L^2(\pi):=\Big\{f\in\B(\mathbb{S}):\sum_{i=1}^N\pi_if_i^2<\8\Big\}.
\end{equation*}
Then $(L^2(\pi),\<\cdot,\cdot\>_0,\|\cdot\|_0)$ is a Hilbert space
with the inner product $\<f,g\>_0:=\sum_{i=1}^N\pi_if_ig_i,f,g\in
L^2(\pi)$. Define the bilinear form $(D(f),\mathscr{D}(D))$ as
\begin{equation*}
D(f):=\ff{1}{2}\sum_{i,j=1}^N\pi_iq_{ij}(f_j-f_i)^2-\sum_{i=1}^N\pi_i\bb_if_i^2,~~~f\in
L^2(\pi),
\end{equation*}
where $\bb_i\in\R,i\in\mathbb{S}$, is given in ({\bf H$2^\prime$}),
and the domain
\begin{equation*}
\mathscr{D}(D):=\{f\in L^2(\pi):D(f)<\8\}.
\end{equation*}
The principal eigenvalue $\lambda_0$ of $D(f)$ is defined by
\begin{equation*}
\lambda_0:=\inf\{D(f):f\in\mathscr{D}(D),\|f\|_0=1\}.
\end{equation*}
For more details on the first eigenvalue, refer to \cite[Chapter
3]{Ch}. Due to the fact that the state space of $\{\LL_t\}_{t\ge0}$
is finite, there exists $\xi=(\xi_1,\cdots,\xi_N)\in\mathscr{D}(D)$
such that
\begin{equation}\label{***}
D(\xi)=\lambda_0\|\xi\|^2_0.
\end{equation}
 Define the operator
\begin{equation*}
\bar\OO :=Q+\mbox{diag}(\bb_1,\cdots,\bb_N),
\end{equation*}
where $Q$ is the $Q$-matrix of $\{\LL_t\}_{t\ge0},$ and $\bb_i\in\R$
such that ({\bf H$2^\prime$}).

\smallskip

The main result in this section is the following.

\begin{thm}\label{eigen}
{\rm Let  ({\bf H1}) and ({\bf H$2^\prime$})  hold, and assume
further $\lambda_0>0$.  Then, the solution of  \eqref{eq1} is almost
surely asymptotically stable.
  }
\end{thm}

\begin{proof}
  Recalling
\eqref{***} and checking  the argument of \cite[Theorem 3.2]{SX},
one has
\begin{equation*}
\xi\gg{\bf0} ~~\mbox{ and
}~~(Q\xi)(i)+\bb_i\xi_i=-\lambda_0\xi_i,~i\in\mathbb{S}.
\end{equation*}
Then, we complete the proof by carrying out an  argument of Theorem
\ref{T-1}.
\end{proof}

Next, an example is constructed to show Theorem \ref{eigen}.

\begin{exa}\label{ex4.3}
{\rm  Let $\{W_t\}$ be a scalar Brownian motion. Consider a
regime-switching diffusion process $(X_t,\LL_t)$, in which
$\{X_t\}_{t\ge0}$ satisfies a scalar SDE
\begin{equation}\label{*6}
\d X_t=b(X_t,t,\LL_t) \d t+\si(X_t,t,\LL_t) \d
W(t),~~X_0=x_0,~~\LL_0=0,
\end{equation}
and $\{\LL_t\}_{t\ge0}$,   independent of $\{W_t\}_{t\ge0}$,  is a
right-continuous Markovian chain taking values in
$\mathbb{S}:=\{0,1,2\}$
\begin{equation*}
Q= \left(\begin{array}{ccc}
  -b & b &0\\
  2a & -2(a+b) & 2b\\
  0 & 3a &-3a
  \end{array}
  \right)
  \end{equation*}
for  $a,b>0$ such that
\begin{equation}\label{*8}
b\in(0,1/4)~~~~\mbox{ and }~~~~a-b>1.
\end{equation}
In \eqref{*6}, for any $(x,t)\in\R\times\R_+,$ let
\begin{align*}
& b(x,t,0)=-x/8, ~~~ ~~~~~~~~~~~~~~~~~~\si(x,t,0)=(1+t)^{-1},\\
& b(x,t,1)=\sin x/(1+t) , ~~~ ~~~~~~~~~\si(x,t,1)=x,\\
& b(x,t,2)= \e^{-t}-x/6-2x^3, ~~~ ~~~~\si(x,t,2)= (x \sin t)/2.
\end{align*}
Hence, one can take $\beta_0=-1/4, \beta_1=1+\vv$, $\bb_2=-1/3+\dd$
for sufficiently small $\vv,\dd>0$, and $\gamma_t=c_{\vv,\dd}
((1+t)^{-2}+\e^{-2t})$ for some $c_{\vv,\dd}>0$. Moreover, by the
notion of $\OO$, for $\xi_i=i+1$, $i=0,1,2,$ we deduce that
\begin{equation*}
\begin{split}
(\bar \OO\xi)(0)=-(-b-\bb_0)\xi_0, ~~ (\bar \OO\xi)(1)=-(a-b-\bb_1)\xi_1,~~
(\bar \OO\xi)(2)=-(a-\bb_2)\xi_2.
\end{split}
\end{equation*}
Due to \eqref{*8}, we can chose $\vv,\dd>0$ sufficiently small such
that
\begin{equation*}
\lambda=\min\{-b-\bb_0,a-b-\bb_1,a-\bb_2\}>0.
\end{equation*}
Thus, one finds that
\begin{equation*}
(\bar \OO\xi)(i)\le-\lambda\xi_i,~~~i=0,1,2.
\end{equation*}
Then $\lambda_0>0$ due to \cite[Theorem 4.4]{SX}. As a result, by
Theorem \ref{M-1}, the solution of  \eqref{*6} is almost surely
asymptotically stable.

\smallskip
Note that the unique stationary distribution of $\{\LL_t\}_{t\ge0}$
is $ \mu=(1/3,1/3,1/3), $ and it is easy to see that  the averaging
condition \eqref{*4} does not  hold no matter how small $\vv,\dd>0$
is. Hence Corrolary \ref{cor-2} does not apply to \eqref{*4}.
Consequently,  Theorem
\ref{eigen} has its own right. }
\end{exa}

\section{Almost Sure Asymptotic Stability: Countable State
Space}\label{sec5} For the case of finite state space (i.e. $N<\8$),
we adopt the Perron-Frobenius theorem and the principal eigenvalue
approach to study almost sure asymptotic stability for
regime-switching diffusion process determined by \eqref{eq1} and
\eqref{love}, respectively. With regard to the first method, the
averaging condition (see \eqref{cond-2}) plays an important role in
the stability analysis. Therefore, one has to provide an explicit
formula of stationary distribution for an irreducible Markov chain
to guarantee the averaging condition to hold under some appropriate
conditions. So this approach seems hard to be generalized to the
case of infinite state space (i.e. $N=\8$) since the explicit
expression of stationary distribution is hard to be obtained. On the
other hand, the principal  eigenvalue approach can also be extend to
the case of infinite state space, however, under an additional
condition that $\lambda_0$ is attainable, i.e., there exists $f\in
L^2(\pi),f\neq0$, such that $D(f)=\lambda_0\|f\|_0^2.$ In this
section, for the case of infinite state space,  by a finite
partition approach and an $M$-matrix theory, we proceed to discuss
almost sure asymptotic stability for regime-switching diffusion
process determined by \eqref{eq1} and \eqref{love}.

\begin{defn}
{\rm (see e.g. \cite[Definition 2.9, p.67]{MY06}) A square matrix
$A=(a_{ij})_{n\times n}$ is called a nonsingular $M$-matrix if $A$
can be expressed in the form $A=sI-B$ with $B\gg{\bf0}$ and
$s>\mbox{Ria}(B)$, where $I$ is the $n\times n$ identity matrix and
$\mbox{Ria}(B)$ the spectral radius of $B.$ }

\end{defn}

We further suppose that
\begin{equation}\label{r3}
K:=\sup_{i\in\mathbb{S}}\bb_i<\8~~~~\mbox{ and }~~~~
\sup_{i\in\mathbb{S}}(-q_{ii})<\8.
\end{equation}
Let us insert $m$ points in the interval $(-\8,K]$ as follows:
\begin{equation*}
-\8=:k_0<k_1<\cdots<k_m<k_{m+1}:=K.
\end{equation*}
Then, the interval $(-\8,K]$ is divided into $m+1$ sub-intervals
$(k_{i-1},k_i]$  indexed by $i$. Let
\begin{equation*}
F_i:=\{j\in\mathbb{S}:\bb_j\in(k_{i-1},k_i]\},~~i=1,\cdots,m+1.
\end{equation*}
Without loss of generality, we can and do assume that each $F_i$ is
not empty. Then
\begin{equation*}
F:=\{F_1,\cdots,F_{m+1}\}
\end{equation*}
is a finite partition of $\mathbb{S}$. For $i,j=1,\cdots,m+1$, set
\begin{equation*}
q_{ij}^F:=
\begin{cases}
\sup_{r\in F_i}\sum_{k\in
F_j}q_{rk},~~~~~~~~j<i,\\
\inf_{r\in F_i}\sum_{k\in F_j}q_{rk},~~~~~~~~~j>i,\\
-\sum_{j\neq i}q_{ij}^F,~~~~~~~~~~~~~~~~~i=j.
\end{cases}
\end{equation*}
So $Q^F:=(q_{ij}^F)$ is the  $Q$-matrix for some Markov chain with
the state space $\mathbb{S}_0:=\{1,\cdots,m+1\}.$ For
$i=1,\cdots,m+1$, let
\begin{equation*}
\bb_i^F:=\sup_{j\in F_i}\bb_j,~~~~~~~~H_{m+1}:=
\left(\begin{array}{ccccc}
  1 & 1 &1 & \cdots & 1\\
  0 & 1 &1 &   \cdots &1\\
  \vdots & \vdots & \vdots & \cdots & \vdots\\
  0 & 0 & 0 & \cdots & 1\\
  \end{array}
  \right)_{(m+1)\times(m+1)}.
  \end{equation*}

\begin{thm}\label{coun}
{\rm Let $N=\8$, and assume that ({\bf H1}), ({\bf H$2^\prime$}) and
\eqref{r3} hold. Assume further that
\begin{equation*}
-(Q^F+\mbox{diag}(\bb_1^F,\cdots,\bb_{m+1}^F))H_{m+1}
\end{equation*}
is a nonsingular $M$-matrix.  Then, the solution of  \eqref{eq1} is
almost surely asymptotically stable.

}
\end{thm}

\begin{proof}
Since $-(Q^F+\mbox{diag}(\bb_1^F,\cdots,\bb_{m+1}^F))H_{m+1}$ is a
nonsingular $M$-matrix, by \cite[Theorem 2.10, p.68]{MY06} there
exists  a vector $\eta^F:=(\eta_1^F,\cdots,\eta_{m+1}^F)^*\gg{\bf0}$
such that
\begin{equation}\label{w3}
(-\lambda_1^F,\cdots,-\lambda_{m+1}^F)^*:=(Q^F+\mbox{diag}(\bb_1^F,\cdots,\bb_{m+1}^F))H_{m+1}\eta^F\ll{\bf0}.
\end{equation}
Set $\xi^F:=H_{m+1}\eta^F$. By the structure of $H_{m+1}$, it is
trivial to see that
\begin{equation*}
\xi_i^F=\eta^F_{m+1}+\cdots+\eta^F_i,~~i=1,\cdots,m+1.
\end{equation*}
This, together with $\eta^F\gg{\bf0}$, yields that $\xi^F\gg{\bf0}$
and $\xi^F_{i+1}<\xi^F_i,i=1,\cdots,m+1$. Next, we extend the vector
$\xi^F$ to be a vector on $\mathbb{S}$ by setting $\xi_r:=\xi_i^F$
for $r\in F_i$.  Moreover, let
$\phi:\mathbb{S}\mapsto\{1,\cdots,m+1\}$ be a map defined by
$\phi(j):=i$ for $j\in F_i.$ Then, by the definition of $\bb_i^F$,
one has
\begin{equation}\label{r1}
\xi_r=\xi_i^F=\xi_{\phi(r)}^F~\mbox{ and }~ \bb_r\le
\bb^F_{\phi(r)}, ~~~r\in F_i.
\end{equation}
 For any $r\in\mathbb{S}$, there exists
$F_i$ such that $r\in F_i$. Recalling the definition of $q_{ij}^F$
and utilizing $\xi^F_{i+1}<\xi^F_i,i=1,\cdots,m+1$, we derive from
\eqref{r1} that, for $r\in F_i$,
\begin{equation}\label{r2}
\begin{split}
(Q\xi)(r) &=\sum_{k<i}\sum_{j\in
F_k}q_{rj}(\xi_j-\xi_r)+\sum_{k>i}\sum_{j\in
F_k}q_{rj}(\xi_j-\xi_r)\\
&=\sum_{k<i}\sum_{j\in
F_k}q_{rj}(\xi_k^F-\xi_i^F)+\sum_{k>i}\sum_{j\in
F_k}q_{rj}(\xi_k^F-\xi_i^F)\\
&\le\sum_{k<i}q_{ik}^F(\xi_k^F-\xi_i^F)+\sum_{k>i}q_{ik}^F(\xi_k^F-\xi_i^F)\\
&=(Q^F\xi^F)(i)=(Q^F\xi^F)(\phi(r)).
\end{split}
\end{equation}
Observe  from \eqref{w3}-\eqref{r2}  that
\begin{equation*}
\begin{split}
&V(X_t)\xi_{\LL_t}\\ &\le
V(x_0)\xi_{i_0}+\int_0^t\{\gamma_s\xi_{\LL_s}+((Q\xi)(\LL_s)+\bb_{\LL_s}\xi_{\LL_s})V(X_s)\}\d
s+\Gamma_1(t)+\Gamma_2(t)\\
 &\le
V(x_0)\xi_{i_0}+\int_0^t\{\gamma_s\xi_{\phi(\LL_s)}^F+((Q^F\xi^F)(\phi(\LL_s))+
\bb^F_{\phi(\LL_s)}\xi_{\phi(\LL_s)}^F)V(X_s)\}\d
s+\Gamma_1(t)+\Gamma_2(t)\\
&\le  |x_0|^2\xi_{i_0}+\mu\int_0^t\gamma_s\d s-\nu\int_0^tV(X_s)\d
s+\Gamma_1(t)+\Gamma_2(t),
\end{split}
\end{equation*}
where $\mu:=\max_{i\in\mathbb{S}_0}\lambda_i^F$,
$\nu:=\min_{i\in\mathbb{S}_0}\lambda_i^F$ due to \eqref{w3}, and
\begin{equation*}
\Gamma_1(t):=2\int_0^t\xi_{\phi(\LL_s)}^F\<\nn
V(X_s),\si(X_s,s,\LL_s)\d W_s\>,~\Gamma_2(t):=\int_0^t\int_\R
V(X_s)\{\xi_{\phi(i_0+h(\LL_s,z))}^F-\xi_{\phi(\LL_s)}^F\}\tt N(\d
s,\d z).
\end{equation*}
Then, the desired assertion follows by imitating an argument of
Theorem \ref{T-1}.
\end{proof}

Before the end of this paper, an example is established to
demonstrate Theorem \ref{eigen}.

\begin{exa}
{\rm Let $\{X_t\}_{t\ge0}$ satisfy a scalar SDE
\begin{equation}\label{*9}
\d X_t=b_{\LL_t}X_t\d t+X_t^2\wedge|X_t|\d W_t,~~~X_0=x_0\neq0
\end{equation}
and $\{\LL_t\}_{t\ge0}$ is a birth-death process on
$\mathbb{S}:=\{1,2,\cdots\}$ with $q_{ii+1}=c_i>0,$
$q_{ii-1}=a_i>0$, and $q_{ij}=0$ for $|i-j|>1$. Let $b_1=-1,
b_i=\kk-i^{-1}, i\ge 2,$ for some $\kk>0,$  and assume
$\sup_{i\in\mathbb{S}}(-q_{ii})<\8$ with $q_{ii}=-(a_i+c_i)$. Note
that
\begin{equation*}
\mathscr{L}_t^{(i)}(|x|)=b_i|x|,~~~i\in\mathbb{S}.
\end{equation*}
Thus  \eqref{r3} holds for $\bb_i=b_i$, and  $\bb_1^F=-1,
\bb_2^F=\kk$. Set $F_1=\{1\}$ and $F_2=\{2,3,\cdots\}$. Due to
$q_{ij}=0$ for $|i-j|>1$. Then we have
\begin{equation*}
-Q^F=\left(\begin{array}{cc}
  -c_1 & c_1\\
  a_2 & -a_2\\
  \end{array}
  \right).
\end{equation*}
As a consequence,  if $\kk\in(0,a_2/(1+c_1)),$  the solution of
\eqref{*9} is almost surely asymptotically stable since $
-(Q^F+\mbox{diag}(\bb_1^F,\bb_2^F))H_2$ is a nonsingular $M$-matrix.
While, under the same condition, Shao and Xi \cite[Example
2.1]{SX14} shows that the solution of \eqref{*9} is asymptotically
stable in probability.
 }
\end{exa}

\section{Stabilization of linear regime-switching diffusions}

Let us now consider the following linear regime-switching diffusions:
\begin{equation}\label{lin-1}
\d X_t=A_{\LL_t}X_t \d t+ \sum^{m}_{k=1} C^{(k)}_{\LL_t}X_t \d W^{(k)}_t,
\end{equation}
on$t\ge 0$ with initial data $x(0)=x_0\in \R^n$ and $\LL_0=i_0\in \mathbb{S}$.

We are required to design a state feedback control $u_t$ in the drift part such that
the corresponding controlled system
\begin{equation}\label{cont-1}
\d X_t=[A_{\LL_t}X_t+B_{\LL_t}u_t] \d t+ \sum^{m}_{k=1} C^{(k)}_{\LL_t} X_t \d W^{(k)}_t,
\end{equation}
becomes almost surely asymptotically stable. Here, the control
$u_t$ is an $\R^{l}$-valued. For each mode $\LL_t=i\in \mathbb{S}:=\{1,2,\cdots, N\},(N<\8)$.
we write $A_{\LL_t}=A_i$, etc. for simplicity, and $A_i,C^{(k)}_{i}$ are $n\times n$ constant matrices while $B_i$
is an $n\times l$ matrix.

Let the linear state feedback control
$
 u_t=K_{\LL_t}X_t
$
depending on the state $X_t$ and Markov chain $\LL_t$,
where $K_{\LL_t}$ is an $l\times n$ matrix.
Hence the closed-loop system becomes
\begin{equation}\label{clos-1}
\d X_t=[A_{\LL_t}X_t+B_{\LL_t}K_{\LL_t}X_t] \d t+ \sum^{m}_{k=1} C^{(k)}_{\LL_t} X_t\d W^{(k)}_t,
\end{equation}

In this section, we give the following result which is used to stabilise \eqref{lin-1} by designing a state feedback
controller in terms of the solutions of a sets of linear matrix inequalities (LMIs), and the average condition.
\begin{thm}\label{stab-0}
{\rm  If there exists a positive definite matrix $\Gamma,$ and matrix $Y_i$, real number $\alpha_i$, such that
the following LMIs hold
\begin{equation}\label{eq-x1}
\left(
  \begin{array}{cc}
    \Phi_{i} & * \\
    \Xi_i & \Theta \\
  \end{array}
\right)<0, ~~ i \in \mathbb{S}
\end{equation}
where
$
\Phi_{i}= (A_i\Gamma+B_iY_i)+(A_i\Gamma+B_iY_i)^T -\alpha_i X, ~
\Xi_i=[(C^{(1)}_{i}\Gamma)^T,(C^{(2)}_{i}\Gamma)^T,
\cdots, (C^{(m)}_{i}\Gamma)^T]^T,
\Theta=\mbox{diag}(-\Gamma,-\Gamma,\cdots,-\Gamma) ,
$
and,
\begin{equation}\label{eq-x2}
\sum_{i=1}^N\mu_i\alpha_i <0,
\end{equation}
where the symbol $*$ denotes the transposed element at the symmetric position.
 Then, the solution of \eqref{clos-1} is almost surely asymptotically stable with respect to state feedback gain
 $K_i=Y_i\Gamma^{-1}.$
}
\end{thm}

Before proceeding further, we give the following lemma which will be used in the proof.
\begin{lem}\label{Sc-C}
{\rm (\cite{BGF})
Let $M, N, P$ be constant matrices with appropriate dimensions such that
$P=P^T>0$ and $M=M^T$. Then $M+NP^{-1}N^T<0$ if\/f
$$\left(
  \begin{array}{cc}
    M & N \\
    N^T & -P \\
  \end{array}
\right)<0.
$$
}
\end{lem}

\noindent
{\it Proof of Theorem \ref{stab-0}.}
Set $P=\Gamma^{-1}$, let $V(x)=x^TPx,$
 we computer
\begin{equation*}
\begin{split}
L^{(i)}V(X_t)= & 2 X_t^T P[(A_i+B_iK_i)X_t] + X_t^T \left\{\sum^m_{k=1} (C^{(k)}_{i})^T P C^{(k)}_{i}\right\}X_t\\
=&  X_t^T [P(A_i+B_iK_i)+(A_i+B_iK_i)^T P ]X_t \\
&+ X_t^T \left\{\sum^m_{k=1}  (C^{(k)}_{i})^T P C^{(k)}_{i}\right\}X_t.
\end{split}
\end{equation*}
On the other hand, by Lemma \ref{Sc-C},
we have from \eqref{eq-x1} that
\begin{equation}\label{eq-x30}
\begin{split}
&[(A_i\Gamma+B_iY_i)+(A_i+B_iY_i)^T -\alpha_i \Gamma ] \\
&\qquad + \sum^m_{k=1} (C^{(k)}_{i}\Gamma)^T \Gamma^{-1} C^{(k)}_{i}\Gamma <0,~~ i\in \mathbb{S}.
\end{split}
\end{equation}
And, noting that $ \Gamma= P^{-1}$ and $Y_i=K_i \Gamma$,
pre- and postmultiplying \eqref{eq-x30} by $P $ results in
\begin{equation}\label{eq-x3}
\begin{split}
&[P(A_i+B_iK_i)+(A_i+B_iK_i)^T P ] \\
&\qquad + \sum^m_{k=1} (C^{(k)}_{i})^T P C^{(k)}_{i} - \alpha_i P<0,~~ i\in \mathbb{S}.
\end{split}
\end{equation}
Thus, \eqref{eq-x3} and \eqref{eq-x2} imply that the conditions \eqref{cond-1} and \eqref{cond-2} hold.
So, the required assertion now follows by Theorem \ref{T-1}.

\begin{rem}
{\rm We should also consider the case that the state feedback control is designed in  the diffusion part, however the results are similar, we omit it here.
 }
\end{rem}

Let us discuss an example to illustrate our theory.
\begin{exa}\label{Ex3}
{\rm
Let $W_t$ be a scalar Brownian motion. Let $\LL_t$ be a right-continuous Markov chain taking values in $\mathbb{S}=\{1,2\}$
with generator
$$
\left(
  \begin{array}{cc}
    -1 & 1 \\
    1 & -1 \\
  \end{array}
\right).
$$
Assume that $W_t$ and $\LL_t$ are independent.
Consider a
regime-switching diffusion process $(X_t,\LL_t)$, in which
$\{X_t\}_{t\ge0}$ satisfies a two-dimensional SDE
\begin{equation}\label{eq-x6}
\d X_t= A_{\LL_t}X_t \d t+C_{\LL_t}X_t \d
W_t,~~~X_0=x_0,~~~\LL_0=1,
\end{equation}
on $t\ge 0$, where
$$
 A_1=\left(
   \begin{array}{cc}
     3 & -1 \\
     1 & -4\\
   \end{array}
 \right),~~
A_2= \left(
             \begin{array}{cc}
               -3 & -1 \\
               1 & 2 \\
             \end{array}
           \right),
$$
$$
 C_1=\left(
   \begin{array}{cc}
     1 & 1 \\
     1 & -1\\
   \end{array}
 \right),~~
C_2= \left(
             \begin{array}{cc}
               -1 & -1 \\
               -1 & 1 \\
             \end{array}
           \right),
$$
 By \cite[Theorem 4.3, p.1045]{Kh07}), the solution of \eqref{eq-x6} is unstable (see, Figure 1 and  Figure 2).
Now, we design the linear state feedback control $u_t=K_{\LL_t} X_t$ to
stabilise \eqref{eq-x6}. For this purpose, we consider the controlled system of the form
\begin{equation}\label{eq-x7}
\d X_t=[A_{\LL_t}X_t+B_{\LL_t}K_{\LL_t}X_t] \d t+C_{\LL_t} X_t\d W_t,
\end{equation}
where
$
 B_1=\left(
       \begin{array}{c}
         -10 \\
         0 \\
       \end{array}
     \right)
 ,~~
B_2= \left(
       \begin{array}{c}
         0 \\
         -10 \\
       \end{array}
     \right).
$
\par
It is easy to see that the stationary distribution of $\LL_t$
is $(\frac{1}{2}, \frac{1}{2})$, and  the average condition of \eqref{eq-x2} becomes
\begin{equation}\label{eq-x8}
 \frac{1}{2}\alpha_1+ \frac{1}{2} \alpha_2 <0.
\end{equation}
\par
This, together with  the  LMIs
\begin{equation*}
\left(
  \begin{array}{cc}
    (A_i\Gamma+B_iY_i)+(A_i\Gamma+B_iY_i)^T -\alpha_i \Gamma & * \\
    C_i\Gamma & -\Gamma \\
  \end{array}
\right)<0, ~~ i \in \mathbb{S},
\end{equation*}
 we obtain
$$
 \alpha_1=-2.9074,\alpha_2= 1.4537,
$$
$$
\Gamma=\left(
    \begin{array}{cc}
  0.1543 &  -0.0007 \\
   -0.0007  &  0.1406 \\
    \end{array}
  \right),
$$
and,
$$
Y_1=(0.0882,0.0017), ~~Y_2=(0.0018, 0.0656).
$$
By Theorem \ref{stab-0}, we can then conclude that the state feedback gain
$$
K_1=Y_1 \Gamma^{-1}=(0.5716, 0.0150),~~K_2=Y_2 \Gamma^{-1}=(0.0137, 0.4666).
$$
such that the solution of \eqref{eq-x7} is almost surely asymptotically stable.
The computer simulation supports this result clearly (see, Figure 3).
\begin{figure}[htbp]
\centering
\includegraphics[height=8cm,width=14cm]{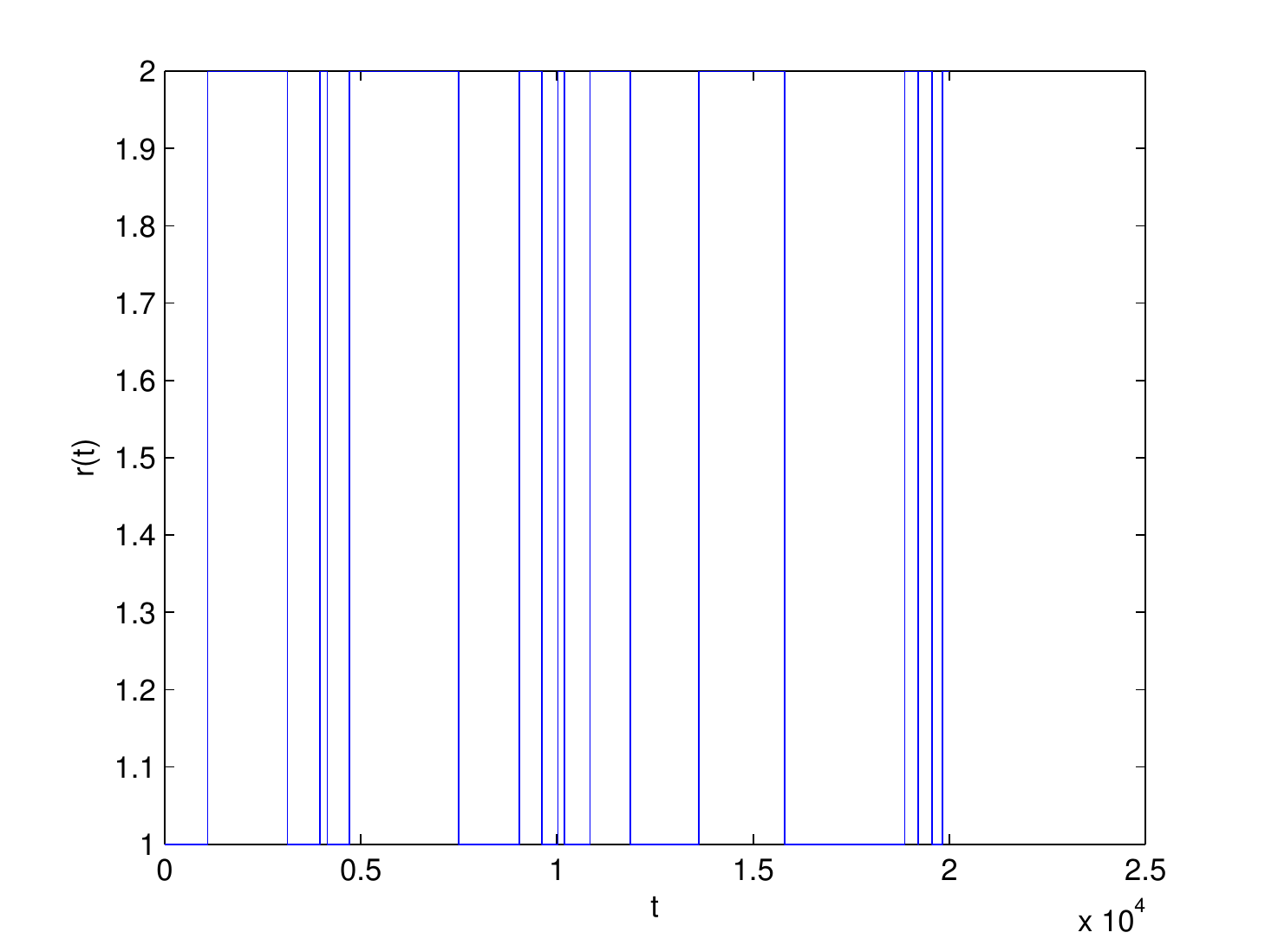}
\caption{ Computer simulation of the paths $\LL_t.$}
\end{figure}

\begin{figure}[htbp]
\centering
\includegraphics[height=8cm,width=14cm]{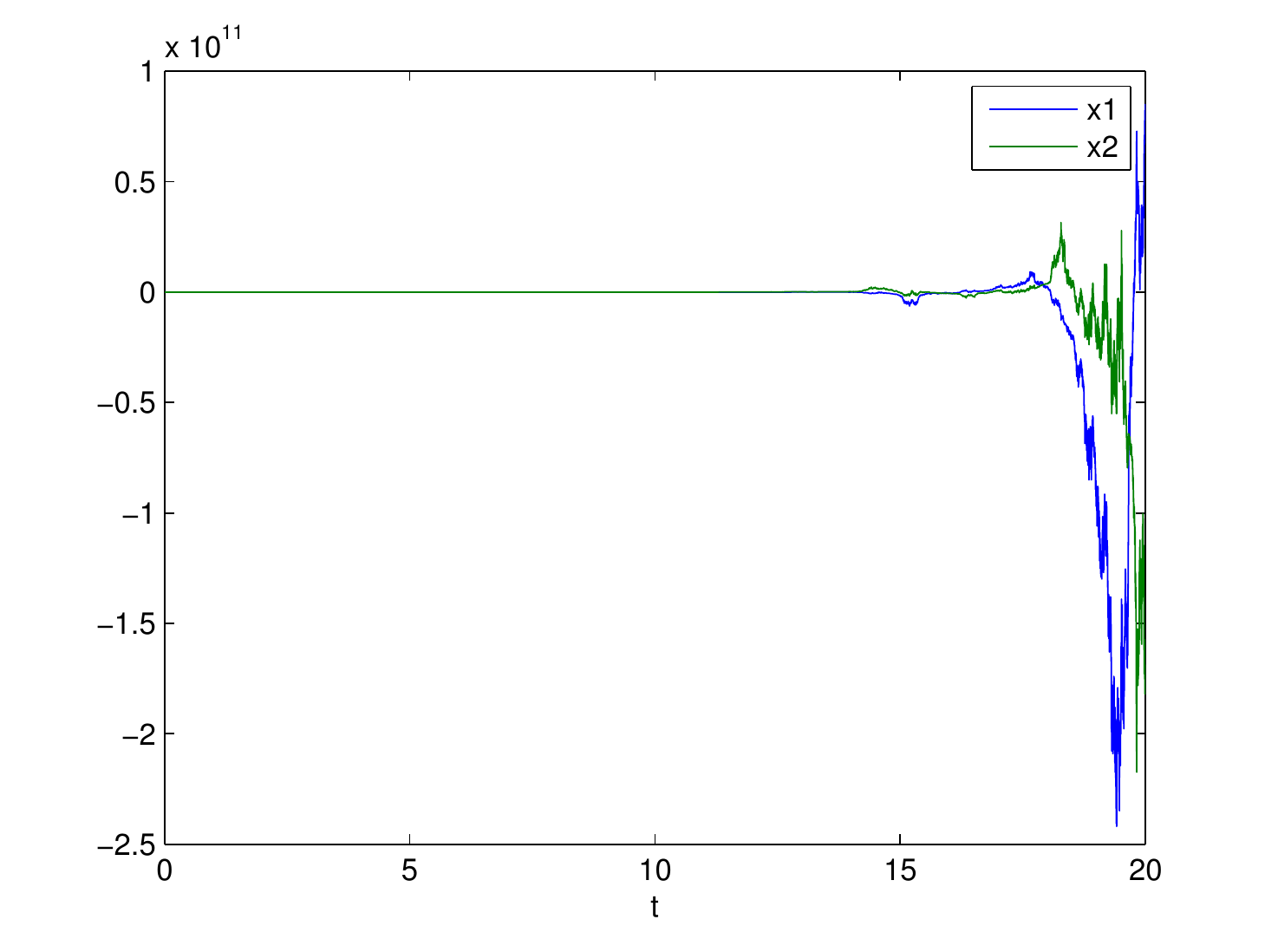}
\caption{Computer simulation of  unstability of $X_t$ in \eqref{eq-x6}. }
\end{figure}

\begin{figure}[htbp]
\centering
\includegraphics[height=8cm,width=14cm]{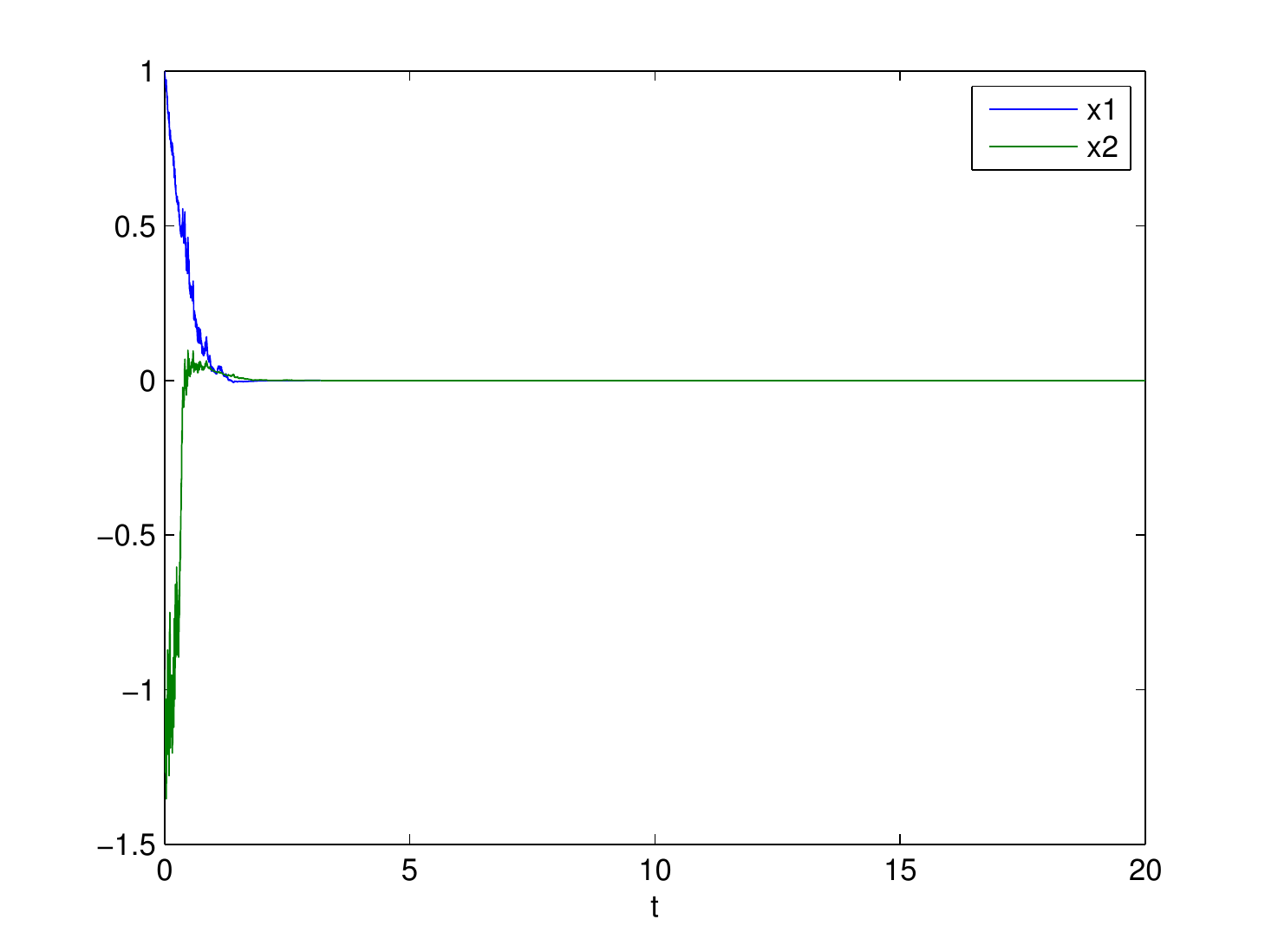}
\caption{Computer simulation of almost sure stabilization of $X_t$ in \eqref{eq-x7}. }
\end{figure}

}
\end{exa}

\newpage

\appendix

\section{Appendix}
Although the latter argument for Theorem 1 is similar to that of
\cite[Theorem 2.1]{YM04},
 we outline the proofs of  \eqref{*021}  and  \eqref{*022} to make the content self-contained.

\smallskip

{\bf Claim \eqref{*021}.} Recall that
\begin{equation}\label{A.2}
\sup_{0\le t<\8}|X_t|<\8~~~\mbox{a.s.}
\end{equation}
and that
\begin{equation}\label{A.3}
\liminf_{t\to \8} V(X_t)=0~~~~~\mbox{a.s.}
\end{equation}%
If  \eqref{*021} is false, then
\begin{equation*}
\P \left\{ \limsup_{t\to \8} V(X_t)>0 \right\}>0.
\end{equation*}
Therefore, there is a number $\vv>0$ such that
\begin{equation}\label{A.5}
\P(\OO_1)\ge 3\vv,
\end{equation}
where $\OO_1=\left\{ \limsup_{t\to \8} V(X_t)>2\vv \right\}$.

By \eqref{A.2}, we find
a positive number $M$, which depends on $\vv$, such that
\begin{equation}\label{A.6}
\P(\OO_2)\ge 1-\vv,
\end{equation}
where
$\OO_2=\left\{ \limsup_{t\to \8} |X_t| < M \right\}$.

From \eqref{A.5} and \eqref{A.6}, we have
\begin{equation}\label{A.7}
\P(\OO_1\cap\OO_2)\ge  2\vv,
\end{equation}
We now define a sequence of stopping times,
\begin{align*}
 & \tau_M:=\inf\{t\ge 0: |X_t|\ge M \},\\
  & \ee_1:=\inf\{t\ge 0: V(X_t)\ge 2\vv \},\\
   & \ee_{2k}:=\inf\{t\ge \ee_{2k-1}: V(X_t)\le \vv \}, ~~k=1,2,\cdots\\
    & \ee_{2k+1}:=\inf\{t\ge \ee_{2k}: V(X_t)\ge 2\vv \},~~k=1,2,\cdots
  \end{align*}
where throughout this paper we set $\inf \emptyset =\8$.  From \eqref{A.3}, we see that
if $\oo\in \OO_1 \cup\OO_2$, then
$\tau_M=\8 $ and $\ee_k <\8, ~~ k=1,2,\cdots. $

 By \eqref{*2}, we get
\begin{align}\label{A.08}
 \8& > \E \int^{\8}_0 V(X_t) \d t \nonumber \\
  & \ge \sum^{\8}_0 \E \left[ {\bf 1}_{\{\ee_{2k-1}<\8, \ee_{2k}<\8, \tau_M=\8\}} \int^{\ee_{2k}}_{\ee_{2k-1}} V(X_t)\d t \right] \nonumber\\
   & \ge \vv \sum^{\8}_0 \E \left[ {\bf 1}_{\{\ee_{2k-1}<\8, \tau_M=\8\}} (\ee_{2k}-\ee_{2k-1})  \right]
  \end{align}
where we use the fact $\ee_{2k}<\8$ whenever $\ee_{2k-1}<\8$.

By ({\bf H1}), we have that there exists a positive constant $L_M$ such that
\begin{equation}\label{A.8}
\begin{split}
& \E \left[  {\bf 1}_{\{\ee_{2k-1} \wedge \tau_M <\8  \}} \sup_{0\le t\le T}
| X_{\tau_M \wedge (\ee_{2k-1}+t)} -  X_{\tau_M \wedge
\ee_{2k-1}} |^2
 \right] \\
 & \le 2L^2_M(T+4)T.
 \end{split}
\end{equation}
From the uniformly continuous of the function $V$ on the bounded closed ball $\bar{B}_M$,
we also choose $ \delta= \delta(\vv)>0$ such that
\begin{equation}\label{A.9}
|V(x)-V(y)|< \vv/2,~~~x,y \in \bar{B}_M, ~~ |x-y|< \delta.
\end{equation}

Furthermore, we choose $T=T(\vv, \delta, M)$ sufficiently small such that
\begin{equation*}
  \frac{2L^2_M(T+4)T}{\delta^2}<\vv.
\end{equation*}
From \eqref{A.8} and \eqref{A.7}, we have
\begin{equation*}
\begin{split}
&\P\bigg( \{\ee_{2k-1}<\8, \tau_M=\8 \}  \\
&\,\,\, \left.
\bigcap  \{ \sup_{0\le t\le T} |X_{ \ee_{2k-1}+t}-X_{ \ee_{2k-1}}|< \delta  \}   \right) \ge \vv.
\end{split}
\end{equation*}

Using \eqref{A.9}, we obtain
\begin{equation}\label{A.10}
\begin{split}
&\P\bigg(
 \{\ee_{2k-1}<\8, \tau_M=\8 \} \\
&\,\,\, \left.
\bigcap  \{ \sup_{0\le t\le T} |V(X_{ \ee_{2k-1}+t})-V(X_{ \ee_{2k-1}})|< \vv  \}   \right) \ge \vv.
\end{split}
\end{equation}

Set
\begin{equation*}
\bar{\OO}_k= \left\{  \sup_{0\le t\le T} |V(X_{ \ee_{2k-1}+t})-V(X_{ \ee_{2k-1}})|< \vv     \right\}.
\end{equation*}
We see
\begin{equation*}
\ee_{2k}(\oo)- \ee_{2k-1}(\oo) \ge T
\end{equation*}
when $ \oo \in \{ \ee_{2k-1}<\8, \tau_M=\8 \} \cap \bar{\OO}_k$.

From \eqref{A.08} and \eqref{A.10}, we derive
\begin{align}\label{A.11}
\begin{split}
 \8 & >    \vv \sum^{\8}_0 \E \left[ {\bf 1}_{\{\ee_{2k-1}<\8, \tau_M=\8\}} (\ee_{2k}-\ee_{2k-1})  \right] \\
   & \ge \vv T  \sum^{\8}_0 \P ( \{\ee_{2k-1}<\8, \tau_M=\8\} \cap  \bar{\OO}_k )  \\
   &\ge \ee T  \sum^{\8}_0 \vv =\8.
 \end{split}
  \end{align}
which is a contradiction. Therefore,  \eqref{*021} must hold.

{\bf Claim \eqref{*022}.} 
If \eqref{*022} is false, there exists some $\bar{\oo} \in \OO_0$
such that
\begin{equation*}
  \limsup_{t\to \8} |X_t(\bar{\oo})|>0 .
\end{equation*}
Then, for some $\varsigma >0,$ there is a subsequence $ X_{t_k}(\bar{\oo})$ of $X_t(\bar{\oo})$
such that $ |X_{t_k}(\bar{\oo})|\ge \varsigma, ~~ k\ge 1 . $

By \eqref{A.2}, there is an increasing subsequence $ X_{\bar{t}_k}(\bar{\oo})$ of $ X_{t_k}(\bar{\oo})$ such that
$\lim_{k\to \8}X_{\bar{t}_k}(\bar{\oo})=z\in \R^n $ with $|z|\ge \varsigma$. Therefore
$\lim_{k\to \8}V(X_{\bar{t}_k}(\bar{\oo}))=V(z)>0. $

But, from \eqref{*021}, we see that this is a contradiction.
Therefore,   \eqref{*022}holds.

\section*{Acknowledgements}
The authors would also like to thank the financial supports from the National Natural Science Foundation of China
(No.61374085, 11401592, 61473213).

\end{document}